\documentclass[12pt,twoside]{article}
\usepackage{amsmath, amsthm, amscd, amsfonts, amssymb, graphicx, color}

\begin{document}

\centerline{}

\centerline{}

\centerline{\Large{\bf Properties of $J$-fusion frames in Krein space}}

\centerline{}

\centerline{\bf {Shibashis Karmakar}}

\centerline{Department of Mathematics }

\centerline{Jadavpur University, Jadavpur-32}

\centerline{West Bengal, India.}

\centerline{\textit{email} : shibashiskarmakar@gmail.com}

\centerline{}

\centerline{\bf {Sk. Monowar Hossein}}

\centerline{Department of Mathematics }

\centerline{Aliah University, IIA/27 New Town, Kolkata-156}

\centerline{West Bengal, India.}

\centerline{\textit{email} : sami$\_ $milu@yahoo.co.uk}

\centerline{}

\centerline{\bf {Kallol Paul}}

\centerline{Department of Mathematics }

\centerline{Jadavpur University, Jadavpur-32}

\centerline{West Bengal, India.}

\centerline{\textit{email} : kalloldada@gmail.com}

\newtheorem{theorem}{Theorem}[section]
\newtheorem{lemma}[theorem]{Lemma}
\newtheorem{proposition}[theorem]{Proposition}
\newtheorem{corollary}[theorem]{Corollary}
\theoremstyle{definition}
\newtheorem{definition}[theorem]{Definition}
\newtheorem{example}[theorem]{Example}
\newtheorem{exercise}[theorem]{Exercise}
\newtheorem{conclusion}[theorem]{Conclusion}
\newtheorem{conjecture}[theorem]{Conjecture}
\newtheorem{criterion}[theorem]{Criterion}
\newtheorem{summary}[theorem]{Summary}
\newtheorem{axiom}[theorem]{Axiom}
\newtheorem{problem}[theorem]{Problem}
\theoremstyle{remark}
\newtheorem{remark}[theorem]{Remark}
\numberwithin{equation}{section}
\centerline{}


\begin{abstract}
In this paper we characterize $\sqrt{2}$-1-uniform $J$-Parseval fusion frames in a Krein space $\mathbb{K}$. We provide a few results regarding construction of new $J$-tight fusion frame from given $J$-tight fusion frames. We also characterize any uniformly $J$-definite subspace of a Krein space $\mathbb{K}$ in terms of a $J$-fusion frame inequality. Finally we generalize the fundamental identity of frames in Krein space $J$-fusion frame setting.
\end{abstract}



\section{Introduction}
Frames for Hilbert space is a well known concept. The theory was introduced by Duffin and Schaeffer \cite{ds} in the year 1952. Now frame theory has applications in almost every areas of applied mathematics. With the development of Hilbert space frame theory and its continuous application in Data transmission, Networking and Signal Processing some emerging research areas created. Filter bank theory, packet-based communication system are some of the examples of new research areas in Frame theory. But to study these areas we need distributed processing such as sensor networks \cite{ssrr}. To tackle this problem frames were build ``locally" and then glued together to obtain frames for the whole space. A beautiful approach was introduced in \cite{pggk} that gives a general method for adding together local frames to get global frames. This powerful method was introduced by Casazza and Kutyniok in \cite{pggk} which they called Frames of subspaces. Later they agree on a terminology of Fusion frames. This concept provides a useful framework in modeling sensor networks \cite{pggkslcj}. Different aspects and applications of fusion frame can be seen in \cite{pgahgku,acmfidgm,pbgkhra,pgmfdmzz,pmdst,pgmfi,pgmfdgywz,pbgkhr,bggk,gkaprctl,pgaksl,akkm,apgkar,mads,pggkslcj,pga}.

Since fusion frame in Hilbert space has such a huge application so it is a natural demand to extend these ideas in Krein space frame theory. Some work already had been done in this direction \cite{akbk,akbkh,ws}. Krein space has some interesting application in modern analysis. Some known areas of application are in High energy physics, Quantam cosmology, Krein space filtering etc. The theory of frames in Krein space can be found in \cite{jb,gmmm,koe,isty,sksk}. P. Acosta-Hum$\acute{a}$nez \cite{paef} defined fusion frames in Krein spaces. In their work they found a correspondence between fusion frames in Hilbert spaces and fusion frames in Krein spaces. But their definition involves fundamental symmetry in Krein space which is not unique. In \cite{sk} we defined fusion frame in Krein spaces in a more geometric setting motivated by the work of Giribet et al. \cite{gmmm}, and we called them $J$-fusion frames.

 In this paper we are interested to investigate the properties of $J$-fusion frame in Krein space frame theory.

\section{Preliminary Notes}
In this section we briefly recall some basic notations, definitions and some important properties useful for our further study. For more detailed information we refer the following references \cite{jb,isty,gmmm,pggk,msak,mads,pga,sksk}.

Let $\pi_M$ be the orthogonal projection from the Hilbert space $\mathbb{H}$ onto the subspace $M$ of $\mathbb{H}$. Then the range space of the projection is $M$ \textit{i.e.} $R(\pi_M)=M$ and the null space of this orthogonal projection is $M^{\perp}$ \textit{i.e.} $N(\pi_M)=M^{\perp}$.
\begin{definition}
Let $I$ be some index set and $\{W_i:i\in{I}\}$ be a family of closed subspaces in $\mathbb{H}$. Also let $\{v_i:i\in{I}\}$ be a family of weights \textit{i.e.} $v_i>0~\forall{i\in{I}}$. Then $\{(W_i,v_i):i\in{I}\}$ is a fusion frame if there exist constants $0<C\leq D<\infty$ such that
	\begin{equation}
	C\|f\|^2 \leq \sum_{i\in I} v_i^2|[\pi_{W_i}f,f]| \leq D \|f\|^2 ~ \textmd{~for every $f\in{\mathbb{H}}$}
	\end{equation}  
\end{definition}
$C$ and $D$ are known as lower and upper bounds respectively for the fusion frame. If $C=D$ then the fusion frame is known as $C$-tight fusion frame and if $C=D=1$ then the fusion frame is known as Parseval fusion frame. Moreover, a fusion frame is called $v$-uniform, if $v:=v_i=v_j$ for all $i,~j \in I$. The family of subspaces $\{W_i:i\in{I}\}$ is an orthonormal basis of subspaces if $\mathbb{H}=\oplus_{i\in{I}}{W_i}$.

Let $\pi_M$ be an orthogonal projection in a Krein space $\mathbb{K}$ onto $M$. Then $\pi_M^2=\pi_M$ and $\pi_M^*=\pi_M$. Here range of the projection \textit{i.e.} $R(\pi_M)=M$ and Null space of $\pi_M$ \textit{i.e.} $N(\pi_M)=M^{\perp}$. Let $Q_{M}$ be the $J$-orthogonal projection onto $M$. But the $J$-orthogonal projection $Q_M$ exists if $M$ is a projectively complete subspace of $\mathbb{K}$. So in general we do not have any relation between $\pi_M$ and $Q_M$. Let $\pi_M^{\#}$ be the $J$-adjoint of $\pi_M$. Then we have $\pi_M^{\#}=J\pi_M{J}$. Also $\pi_{JM}=J\pi_M{J}$. Hence $\pi_{JM}=\pi_M^{\#}$.

Let $\{W_i:i\in I\}$ be a collection of non-indefinite subspaces of the Krein space $\mathbb{K}$. We consider the space $\big(\sum_{i\in{I}}\oplus{W_i}\big)$. Then if $f\in\big(\sum_{i\in{I}}\oplus{W_i}\big)$ then $f=\{f_i\}_{i\in I}$, where $f_i\in W_i$ for each $i\in I$. Let $I_+=\{i\in{I}:[f_i,f_i]\geq 0~\textmd{for all~} f_i\in{W_i}\}$ and $I_-=\{i\in{I}:[f_i,f_i]<0~\textmd{for all~} f_i\in{W_i}\}$. We define $[f,g]=\sum_{i\in{I}}[f_i,g_i]$, where $f,g\in\big(\sum_{i\in{I}}\oplus{W_i}\big)$. If the series is unconditionally convergent then $[\cdot,\cdot]$ defines an inner product on $\big(\sum_{i\in{I}}\oplus{W_i}\big)$. 
\subsection{Definition of $J$-fusion frame}
Let $I$ be some index set and let $\{v_i:i\in{I}\}$ be a family weights \textit{i.e.} $v_i>0~\forall~{i\in{I}}$. Let $\{W_i:i\in{I}\}$ be a Bessel family of closed non-indefinite subspaces of a Krein space $\mathbb{K}$ with synthesis operator $T_{W,v}\in{L\Big(\big(\sum_{i\in{I}}\oplus{W_i}\big)_{\ell_2},\mathbb{K}\Big)}$. Let $I_+=\{i\in{I}:[f_i,f_i]\geq 0~\textmd{for all~} f_i\in{W_i}\}$ and $I_-=\{i\in{I}:[f_i,f_i]<0~\textmd{~for all~} f_i\in{W_i}\}$. Now consider the orthogonal decomposition of $\big(\sum_{i\in{I}}\oplus{W_i}\big)_{\ell_2}$ given by
$$\big(\sum_{i\in{I}}\oplus{W_i}\big)_{\ell_2}=\big(\sum_{i\in{I_+}}\oplus{W_i}\big)_{\ell_2}\bigoplus{\big(\sum_{i\in{I_-}}\oplus{W_i}\big)_{\ell_2}},$$
and denote by $P_{\pm}$ the orthogonal projection onto $(\sum_{i\in{I_{\pm}}}\oplus{W_i})_{\ell_2}$. Also, let ${T_{W,v}}_{\pm}=T_{W,v}P_{\pm}$. If $M_{\pm}=\overline{\sum_{i\in{I_{\pm}}}W_i}$, notice that $\sum_{i\in{I_{\pm}}}W_i\subseteq{R({T_{W,v}}_{\pm})}\subseteq{M_{\pm}}$ and
$$R(T_{W,v})=R({T_{W,v}}_+)+R({T_{W,v}}_-)$$

The Bessel family $\mathbb{F}=\{(W_i,v_i):i\in{I}\}$ is a $J$-fusion frame for $\mathbb{K}$ if $R({T_{W,v}}_+)$ is a maximal uniformly $J$-positive subspace of $\mathbb{K}$ and $R({T_{W,v}}_-)$ is a maximal uniformly $J$-negative subspace of $\mathbb{K}$.

Let $\{(W_i,v_i):i\in{I}\}$ be a $J$-fusion frame for $\mathbb{K}$ then $\big(\sum_{i\in{I}}\oplus{W_i},[\cdot,\cdot]\big)$ is a Krein space. The fundamental symmetry let $J_2$ is defined by $J_2(f)=\{f_i:i\in I_+\}\cup\{-f_i:i\in I_-\}$ for all $f$. Also $[f,g]_{J_2}=\sum_{i\in{I_+}}[f_i,g_i]-\sum_{i\in{I_-}}[f_i,g_i]$. Now consider the space $\big(\sum_{i\in{I}}\oplus{W_i}\big)_{\ell_2}=\Big\{f:\big(\sum_{i\in{I}}\oplus{W_i}\big):\sum_{i\in{I}}\|f_i\|_J^2<\infty\Big\}$. We will use this space frequently in our work.
\begin{theorem}
Let $\mathbb{F}=\{(W_i,v_i)\}_{i\in{I}}$ be a $J$-fusion frame for $\mathbb{K}$. Then $\mathbb{F}_{\pm}=\{(W_i,v_i)\}_{i\in{I}_{\pm}}$ is fusion frame for the Hilbert space $(M_{\pm},\pm[\cdot,\cdot])$ \textit{i.e.} there exists constants $B_-\leq{A_-}<0<A_+\leq{B_+}$ such that
\begin{equation}\label{JFFEQ}
A_{\pm}[f,f]\leq\sum_{i\in{I_{\pm}}}v_i^2|[{\pi_{W_i}|}_{M_\pm}(f),f]|~{\leq}~B_{\pm}[f,f]~\textmd{for every }f\in{M_\pm}
\end{equation} 
\end{theorem}
\begin{theorem}
	For each $i\in{I}$, let $v_i>0$ and $\{W_i:i\in{I}\}$ be a collection of definite subspaces also let $\{f_{ij}\}_{j\in{J_i}}$ be a $J$-frame sequence in $\mathbb{K}$. Define $W_i=\overline{span}_{j\in{J_i}}\{f_{ij}\}$ for all $i\in{I}$. Then the following conditions are equivalent.\\
	
	$(i)$ $\{v_if_{ij}\}_{i\in{I},j\in{J_i}}$ is a $J$-frame for $\mathbb{K}$.

	$(ii)$ $\{(W_i,v_i):i\in{I}\}$ is a $J$-fusion frame of subspaces for $\mathbb{K}$.
	
\end{theorem}
Now let $\mathbb{F}=\{(W_i,v_i):i\in{I}\}$ be a $J$-fusion frame for the Krein space $\mathbb{K}$. Then $\{W_i:i\in{I_+}\}$ is a collection of uniformly $J$-positive subspaces of $\mathbb{K}$ and $\{W_i:i\in{I_-}\}$ is a collection of uniformly $J$-negative subspaces of $\mathbb{K}$. Let $T_{W,v}^{\#}$ be the $J$-adjoint operator of the synthesis operator $T_{W,v}$. $T_{W,v}^{\#}$ is called the analysis operator of the $J$-frame of subspaces $F$. Now $T_{W,v}^{\#}=(T_{{W,v}_+}^{\#}+T_{{W,v}_-}^{\#})$. Also we have $N(T_{{W,v}_+}^{\#})^{[\perp]}=R(T_{{W,v}_+})=M_+$. $T_{{W,v}_+}^{\#}(f)=\{v_i\pi_{JW_i}(f)\}_{i\in{I_+}}$ for all $f\in{\mathbb{K}}$. $T_{{W,v}_-}^{\#}(f)=-\{v_i\pi_{JW_i}(f)\}_{i\in{I_-}}$ for all $f\in{\mathbb{K}}$. So $T_{W,v}^{\#}(f)=\{\sigma_iv_i\pi_{JW_i}(f)\}_{i\in{I}}$ for all $f\in{\mathbb{K}}$. Also $T_{{W,v}_+}^{\#}(f)=\{v_i\pi_{W_i}(f)\}_{i\in{I_+}}$ for all $f\in{M_+}$.
Here $\sigma_{i}=1$ if $i\in{I_+}$ and $\sigma_{i}=-1$ if $i\in{I_-}$.

\begin{definition}
The linear operator $S_{W,v}:\mathbb{K}\rightarrow{\mathbb{K}}$ defined by $S_{W,v}(f)=\sum_{i\in I}\sigma_iv_i^2\pi_{J(W_i)}(f)$ is said to be the $J$-fusion frame operator for the $J$-fusion frame $\{(W_i,v_i\}_{i\in I}$.
\end{definition}
\begin{theorem}
If $\{(W_i,v_i)\}_{i\in I}$ is a $J$-fusion frame for the Krein space $\mathbb{K}$ wth synthesis operator $T_{W,v}\in{L\Big(\big(\sum_{i\in{I}}\oplus{W_i}\big)_{\ell_2},\mathbb{K}\Big)}$ then the $J$-frame operator $S_{W,v}$ is bijective and $J$-selfadjoint.
\end{theorem}
\begin{theorem}
If $\{(W_i,v_i)\}_{i\in I}$ is a $J$-fusion frame for the Krein space $\mathbb{K}$ with $J$-fusion frame operator $S_{W,v}$, then $\{(S_{W,v}^{-1}(W_i),v_i)\}_{i\in I}$ is a $J$-fusion frame for $\mathbb{K}$ with $J$-fusion frame operator $S_{W,v}^{-1}$.
\end{theorem}
\section{Main Results}
The idea of Tight/Parseval frames has many important applications in Hilbert space frame theory as well as in Krein space frame theory. In finite dimensional cases the concept has some beautiful geometry. So in this section we are motivated to define Tight $J$-fusion frames in Krein space and further we will study some important properties derived from the definition.

From the definition \ref{JFFEQ} we know that every $J$-fusion frame in $\mathbb{K}$ decomposes the Krein space into two parts. The positive part $M_+$ is the maximal uniformly $J$-positive subspace and the negative part $M_-$ is the maximal uniformly $J$-negative subspaces. Now the cone of neutral vectors in $\mathbb{K}$ is denoted by $\verb"C"$ and defined by $\verb"C"=\{n\in{\mathbb{K}}:[n,n]=0\}$. In \cite{gmmm} we have a concept of angle between any uniformly $J$-definite subspace of $\mathbb{K}$ and $\verb"C"$. Using the results of \cite{gmmm}, we have the following two equations. 
$c_0(M_+,\verb"C")=\frac{1}{\sqrt2}(\sqrt{\frac{1+\alpha^+}{2}}+\sqrt{\frac{1-\alpha^+}{2}})$ and $c_0(M_-,\verb"C")=\frac{1}{\sqrt2}(\sqrt{\frac{1+\beta^+}{2}}+\sqrt{\frac{1-\beta^+}{2}})$, where $\alpha^+=\gamma({G_{M_+}})$ and $\beta^+=\gamma({G_{M_-}})$.

Now every $J$-fusion frame is associated with a positive real numbers $\zeta$, where $\zeta=c_0(M_+,\verb"C")+c_0(M_-,\verb"C")$. We also have $\zeta~\in[\sqrt{2},2)$. We will use the real number $\zeta$ (associated with a $J$-frame for $\mathbb{K}$) extensively in our work and instead of the term $J$-fusion frame for $\mathbb{K}$ we will use the term $\zeta-J$-fusion frame for $\mathbb{K}$.

\begin{definition}
Let $(\mathbb{K},[\cdot,\cdot],J)$ be a Krein space and $\mathbb{F}=\{(W_i,v_i):i\in{I}\}$ be a $J$-fusion frame for the Krein space $\mathbb{K}$. Then $\mathbb{F}$ is said to be a $\zeta-J$-tight fusion frame iff
\begin{equation}
A_{\pm}[f,f]=\sum_{i\in{I_{\pm}}}v_i^2|[{\pi_{W_i}|}_{M_\pm}(f),f]|,~~~~\textmd{ for all}{~f\in{M_{\pm}}}
\end{equation}
Moreover, $\mathbb{F}$ is said to be a $\zeta-J$-Parseval fusion frame if it is a $\zeta-J$-tight frame for the Krein space $\mathbb{K}$ and in addition $A_{\pm}={\pm}1$.
\end{definition}
\begin{example}
Consider the Vector space $\ell_2(\mathbb{R})$. Let us define a inner product on $\ell_2(\mathbb{R})$ by $[x,y]=x_1 y_1 + x_2 y_2-x_3 y_3+x_4 y_4+\ldots$ when $x=(x_1,x_2,x_3,x_4,\ldots), y=(y_1,y_2,y_3,y_4,\ldots)$ and $ x_i, y_i\in\mathbb{R}$ for $i \in{\mathbb{N}}$. Consider the subspaces $W_1=span\\
\{(-\frac{\sqrt3}{2},-\frac{1}{2},0,0,\ldots)\},~W_2=span\{(\frac{\sqrt3}{2},-\frac{1}{2},0,0,\ldots)\},~W_3=span\{(0,1,0,0,\ldots)\}\\
,~W_4=span\{(\frac{1}{\sqrt2},0,\frac{\sqrt3}{\sqrt2},\ldots)\}$. Let $v_1=v_2=v_3=\frac{\sqrt2}{\sqrt3}$ and $v_4=1$. Then $\big\{(W_i,v_i):i\in\{1,2,3,4\}\big\}$ is a $J$-fusion frame for the above Krein space $\ell_2(\mathbb{R})$. It is also a $J$-Parseval fusion frame with $\zeta=\frac{3}{2\sqrt2}+\frac{\sqrt3}{2\sqrt2}$. By numerical calculation we have $\alpha^+=1$ and $\beta^+=\frac{1}{2}$.\\
\end{example}

\subsection{Some results on $J$-tight fusion frames}
In this section we will define $J$-orthonormal basis of subspaces for a Krein space $\mathbb{K}$. We will also find a relation between $J$-orthonormal basis of subspaces and $1$-uniform $J$-parseval fusion frames in $\mathbb{K}$. 
\begin{definition}
Let $\mathbb{F}=\{W_i:i\in{I}\}$ be a family of closed non-indefinite subspaces of a Krein space $\mathbb{K}$. Then the family is said to be an $J$-orthonormal basis of subspaces if $\mathbb{K}^+=\dot{\oplus}_{i\in{I_+}}W_i$ and $\mathbb{K}^-=\dot{\oplus}_{i\in{I_-}}W_i$, where $\mathbb{K}=\mathbb{K}^+[\dot{\oplus}]\mathbb{K}^-$ is the cannonical decomposition of the Krein space.
\end{definition}
 Consider $\sqrt{2}-1-\textmd{uniform~}J$-Parseval fusion frames for a Krein space $\mathbb{K}$. Our next theorem will establish a relation between $\sqrt{2}-1-\textmd{uniform~}J$-Parseval fusion frames and $J$-orthonormal basis of subspaces for a given Krein space.
\begin{theorem}
Let $\mathbb{K}$ be a Krein space. Then any $\zeta-J$-fusion frame is a $J$-orthonormal basis of subspaces for $\mathbb{K}$ iff it is a $\sqrt{2}-1-\textmd{uniform~}J$-Parseval fusion frame for $\mathbb{K}$.
\end{theorem}
\begin{proof}
Let $\{W_i:i\in{I}\}$ be a $J$-orthonormal basis of subspaces in $\mathbb{K}$. So let $I_+=\{i\in{I}:[f_i,f_i]>0~\textmd{for all~} f_i\in{W_i}\}$ and $I_-=\{i\in{I}:[f_i,f_i]<0~\textmd{for all~} f_i\in{W_i}\}$. Then $\mathbb{K}^+=\dot{\oplus}_{i\in{I_+}}W_i$ and $\mathbb{K}^-=\dot{\oplus}_{i\in{I_-}}W_i$, where $\mathbb{K}=\mathbb{K}^+[\dot{\oplus}]\mathbb{K}^-$. Now $(\mathbb{K}^+,[\cdot,\cdot])$ is a Hilbert space and $\gamma{(G_{\mathbb{K}^+})}=1$. Hence $c_0(\mathbb{K}^+,\verb"C")=\frac{1}{\sqrt{2}}$. By similar arguments we have $c_0(\mathbb{K}^-,\verb"C")=\frac{1}{\sqrt{2}}$. Hence $\zeta=\sqrt2$. Also it is easy to see that $\{W_i:i\in{I_+}\}$ is an orthonormal basis of subspaces in $\mathbb{K}^+$. Hence it is a $1$-uniform Parseval frame of subspaces in $\mathbb{K}^+$. So, $\{W_i:i\in{I_+}\}\cup\{W_i:i\in{I_-}\}=\{W_i:i\in{I}\}$ is a $\sqrt{2}-1-\textmd{uniform~}J$-Parseval fusion frame for $\mathbb{K}$.

Conversely, let $\{W_i:i\in{I}\}$ be a $\sqrt{2}-1-\textmd{uniform~}J$-Parseval fusion frame for $\mathbb{K}$. So let $I_+=\{i\in{I}:[f_i,f_i]>0~\textmd{for all~} f_i\in{W_i}\}$ and $I_-=\{i\in{I}:[f_i,f_i]<0~\textmd{for all~} f_i\in{W_i}\}$. Consider $M_{\pm}=\overline{\sum_{i\in{I_{\pm}}}W_i}$. Now it is clear that $\{W_i:i\in{I_+}\}$ is a $1$-uniform Parseval fusion frame for $(M_+,[\cdot,\cdot])$ and $\{W_i:i\in{I_-}\}$ is a $1$-uniform Parseval fusion frame for $(M_-,-[\cdot,\cdot])$. Hence from \cite{pggk} we know that $\{W_i:i\in{I_+}\}$ is an orthonormal basis of subspaces for $(M_+,[~.~])$ and $\{W_i:i\in{I_-}\}$ is an orthonormal basis of subspaces for $(M_-,-[\cdot,\cdot])$. Again we have $c_0(M_+,\verb"C")+c_0(M_-,\verb"C")=\sqrt{2}$, which implies that $c_0(M_+,\verb"C")=\frac{1}{\sqrt{2}}$ and $c_0(M_-,\verb"C")=\frac{1}{\sqrt{2}}$. Also by some simple numerical calculation we have $\gamma{(G_{M_+})}=1$ and also $\gamma{(G_{M_-})}=1$. Using all the above results we conclude that $\mathbb{K}=M_+{[\dot{\oplus}]}M_-$.
Hence the proof.
\end{proof}
Let $\mathbb{F}=\{W_i:i\in{I}\}$ be a sequence of non-neutral definite subspaces in a Krein space $(\mathbb{K},[\cdot,\cdot],J)$. Consider $M_+=\overline{\sum_{i\in{I_{+}}}W_i}$ and $M_-=\overline{\sum_{i\in{I_{-}}}W_i}$, where $I_+=\{i\in{I}:[f_i,f_i]>0~\textmd{for all~} f_i\in{W_i}\}$ and $I_-=\{i\in{I}:[f_i,f_i]<0~\textmd{for all~} f_i\in{W_i}\}$. Now if $M_+$ is a maximal uniformly $J$-positive subspace of $\mathbb{K}$  and $M_-$ is a maximal uniformly $J$-negative subspace of $\mathbb{K}$. Then $\{W_i:i\in{I_+}\}$ and $\{W_i:i\in{I_-}\}$ will be fusion frames for $(M_+,[\cdot,\cdot])$ and $(M_-,-[\cdot,\cdot])$ respectively. Let ${T_{W,v}}_1$ be the synthesis operator for the fusion frame $\{W_i:i\in{I_+}\}$ and ${T_{W,v}}_2$ be the synthesis operator for the frame $\{W_i:i\in{I_-}\}$. Let ${T_{W,v}}_1^\ast$ and ${T_{W,v}}_2^\ast$ are adjoint operators of ${T_{W,v}}_1$ and ${T_{W,v}}_2$ respectively.
\begin{definition}
Let $(\mathbb{K},[\cdot,\cdot],J)$ be a Krein space. A collection of subspaces $\mathbb{F}=\{W_i:i\in{I}\}$ is said to be a disjoint sequence of subspaces in $\mathbb{K}$ if \\
 $\overline{\sum_{i\in{I_{+}}}W_i}~{\cap}~\overline{\sum_{i\in{I_{-}}}W_i}=\{0\}$.
\end{definition}
\begin{example}
Every $\zeta-J$-fusion frame in a Krein space $\mathbb{K}$ is a disjoint sequence of $\mathbb{K}$.
\end{example}
\begin{definition}
Let $\mathbb{F}=\{W_i:i\in{I}\}$ be a collection of subspaces in the Krein space $\mathbb{K}$. Then $\mathbb{F}$ is said to be strictly disjoint sequence iff $\overline{\sum_{i\in{I_{+}}}W_i}~{[\perp]}~\overline{\sum_{i\in{I_{-}}}W_i}$.
\end{definition}
We will now derive an useful result regarding $\zeta-J$-Parseval fusion frames for a Krein space $\mathbb{K}$. Our theorem guarantees
that a Krein space is richly supplied with $\zeta-J$-Parseval fusion frames.
\begin{theorem}
Let $(\mathbb{K},[\cdot,\cdot],J)$ be a Krein space. Assume that $M_1$ and $M_2$ are definite subspaces of $\mathbb{K}$ respectively. Let $\{(X_i,u_i)\}_{i\in{I_1}}$ and $\{(Y_i,v_i)\}_{i\in{I_2}}$ are Parseval fusion frames for $M_1$ and $M_2$ respectively.Then $\{X_i\}_{i\in{I_1}}\cup\{Y_i\}_{i\in{I_2}}$ is a strictly disjoint sequence in $\mathbb{K}$ only if $\{(X_i,u_i)\}_{i\in{I_1}}\cup\{(Y_i,v_i)\}_{i\in{I_2}}$ is a $\zeta-J$-Parseval fusion frame for $\mathbb{K}$ for $\zeta{\in}[\sqrt{2},2)$.
\end{theorem}
\begin{proof}
Without any loss of generality we assume that $M_1$ is positive $J$-definite. Hence it is intrinsically complete. Hence $(M_1,[\cdot,\cdot])$ is a Hilbert space. Given that $\{(X_i,u_i):i\in{I_1}\}$ is a Parseval fusion frame for $(M_1,[\cdot,\cdot])$. Therefore, $\overline{\sum_{i\in{I_{+}}}X_i}=M_1$ and also $X_i\in\mathbb{P}^+$.

Now since $\{X_i\}\cup\{Y_i\}$ is a strictly disjoint sequence in $\mathbb{K}$, so $\overline{\sum_{i\in{I_{1}}}X_i}~{[\perp]}~$\\
$\overline{\sum_{i\in{I_{2}}}Y_i}=\{0\}$. Hence $\overline{\sum_{i\in{I_{2}}}Y_i}$ is a closed negative subspace of $\mathbb{K}$. Now $\{(Y_i,v_i):i\in{I_2}\}$ is a Parseval fusion frame for $M_2$. Therefore, $\overline{\sum_{i\in{I_{2}}}Y_i}=M_2$. Hence $M_2$ is a negative $J$-definite subspace of $\mathbb{K}$. Since $M_1~[\perp]~M_2$, so we have a fundamental decomposition of $\mathbb{K}$ \textit{i.e.} $\mathbb{K}=M_1~[\dot{+}]~M_2$ (see \cite{isty}). Since both $M_1$ and $M_2$ is closed $J$-definite and also intrinsically complete, hence both $M_1$ and $M_2$ are uniformly $J$-definite (see \cite{jb}). Let $\zeta=c_0(M_+,\verb"C")+c_0(M_-,\verb"C")$, then $\{(X_i,u_i)\}_{i\in{I_1}}\cup\{(Y_i,v_i)\}_{i\in{I_2}}$ is a $\zeta-J$-fusion frame for $\mathbb{K}$. As both $\{(X_i,u_i)\}_{i\in{I_1}}$ and $\{(Y_i,v_i)\}_{i\in{I_2}}$ are Parseval fusion frames for $M_1$ and $M_2$ respectively, so $\{(X_i,u_i)\}_{i\in{I_1}}\cup\{(Y_i,v_i)\}_{i\in{I_2}}$ is also a $\zeta-J$-Parseval fusion frame for $\mathbb{K}$.
\end{proof}
\begin{remark} The statement of the above theorem is sufficient but not necessary. Since let $\{(X_i,u_i)\}_{i\in{I_1}}\cup\{(Y_i,v_i)\}_{i\in{I_2}}$ is a $\zeta-J$-Parseval fusion frame for $\mathbb{K}$ for $\zeta{\in}[\sqrt{2},2)$, then there exists uniformly $J$-positive definite subspace $M_+$ and uniformly $J$-negative subspace $M_-$ such that $\{(X_i,u_i)\}_{i\in{I_1}}$ is a Parseval fusion frame for $(M_+,[\cdot,\cdot])$ and $\{(Y_i,v_i)\}_{i\in{I_2}}$ is a Parseval fusion frame for the Hilbert space $(M_-,-[\cdot,\cdot])$. But $M_+$ may not be perpendicular to $M_-$. Hence $\{X_i\}_{i\in{I_1}}\cup\{Y_i\}_{i\in{I_2}}$ may not be a strictly disjoint sequence in $\mathbb{K}$.
\end{remark}
The following result for $\zeta-J$-Parseval fusion frames describes how $\zeta-J$-Parseval fusion frames can be combined to form a new $\zeta-J$-Parseval fusion frame under some restrictions.
\begin{theorem}
Let $\{(X_i,v_i):i\in{I}\}$ and $\{(Y_i,v_i):i\in{I}\}$ are $\zeta-J$-Parseval fusion frames for $\mathbb{K}$ such that $M_+=\overline{\sum_{i\in{I_{+}}}Y_i}=\overline{\sum_{i\in{I_{+}}}X_i}$ and $M_-=\overline{\sum_{i\in{I_{-}}}Y_i}=\overline{\sum_{i\in{I_{-}}}X_i}$. Also let $\{X_i:i\in{I}\}$ and $\{Y_i:i\in{I}\}$ are strictly disjoint sequence in $\mathbb{K}$. Then $\{(X_i+Y_i,v_i):i\in{{I}}\}$ is a $\zeta-J$-Parseval fusion frame for $\mathbb{K}$ iff ${T^{\ast}_{X,v}}^+{T_{Y,v}}^++{T^{\ast}_{Y,v}}^+{T_{X,v}}^+=0$ and ${T^{\ast}_{X,v}}^-{T_{Y,v}}^-+{T^{\ast}_{Y,v}}^-{T_{X,v}}^-=0$, where ${T_{X,v}}^+$ and ${T_{Y,v}}^+$ are synthesis operators of $\{(X_i,v_i):i\in{I_+}\}$ and $\{(Y_i,v_i):i\in{I_{+}}\}$ respectively and ${T_{X,v}}^-$ and ${T_{Y,v}}^-$ are synthesis operators of $\{(X_i,v_i):i\in{I_-}\}$ and $\{(Y_i,v_i):i\in{I_{-}}\}$ respectively.
\end{theorem}
\begin{proof}
Since $\{X_i:i\in{I}\}$ and $\{Y_i:i\in{I}\}$ are strictly disjoint sequence in $\mathbb{K}$ hence $X_i+Y_i$ are closed subspaces for each $i\in{I}$. Let $\{(X_i+Y_i,v_i):i\in{{I}}\}$ be a $\zeta-J$-Parseval fusion frame for $\mathbb{K}$. Also it is given that $\{(X_i,v_i):i\in{I}\}$ and $\{(Y_i,v_i:i\in{I}\}$ are $\zeta-J$-Parseval fusion frames for $\mathbb{K}$. Hence $X_i\subset M_+$ for all $i\in{I_+}$. Similarly $Y_i\subset M_+$ for all $i\in{I_+}$. According to our assumption for each $i\in{I_+}$, $X_i$ is uniformly $J$-positive subspace of $\mathbb{K}$. Hence regular. Since $X_i+Y_i$ is a closed subspace of $M_+$. Hence it is also regular. Similarly we can say that for each $i\in{I_-}$, $X_i+Y_i$ is also regular.

Since $\{(X_i,v_i):i\in{I}\}$ is a $\zeta-J$-Parseval fusion frame for $\mathbb{K}$, therefore $\{(X_i,v_i):i\in{I_+}\}$ is a Parseval fusion frame for $(M_+,[\cdot,\cdot])$. Let ${T_{X,v}}^+$ be the synthesis operator of $\{(X_i,v_i):i\in{I_+}\}$ for the space $(M_+,[\cdot,\cdot])$. Then the synthesis operator ${T_{X,v}}^+$ is defined by ${T_{X,v}}^+(\{f_i\})=\sum_{i\in{I_+}}v_if_i$, where $f_i\in{X_i}$ and ${T^{\ast}_{X,v}}^+$, the analysis operator is defined by ${T^{\ast}_{X,v}}^+(f)=\{v_i{\pi_{W_i}|}_{M_+}(f)\}_{i\in{I_+}}$.

Similarly $\{(Y_i,v_i):i\in{I}\}$ is also a $\zeta-J$-Parseval fusion frame for $\mathbb{K}$. Hence $\{(Y_i,v_i):i\in{I_{+}}\}$ is a Parseval frame for $(M_+,[\cdot,\cdot])$. So, ${T_{Y,v}}^+$ and ${T^{\ast}_{Y,v}}^+$ are defined as above. Also $\{(X_i,v_i):i\in{I_-}\}$ and $\{(Y_i,v_i):i\in{I_{-}}\}$ are Parseval frames for $(M_-,-[\cdot,\cdot])$, so we can define the operators \textit{viz.} ${T_{X,v}}^-$, ${T^{\ast}_{X,v}}^-$, ${T_{Y,v}}^-$ and ${T^{\ast}_{Y,v}}^-$ as above.

Now $\{(X_i+Y_i,v_i):i\in{{I}}\}$ is a $\zeta-J$-Parseval fusion frame for $\mathbb{K}$. So $\{(X_i+Y_i,v_i):i\in{I_+}\}$ is a Parseval fusion frame for $(M_+,[\cdot,\cdot])$. Similarly $\{(X_i+Y_i,v_i):i\in{I_-}\}$ is a Parseval fusion frame for $(M_-,-[\cdot,\cdot])$. Let ${T_{X+Y,v}}^+$ be the synthesis operator for the fusion frame $\{(X_i+Y_i,v_i):i\in{{I_+}}\}$ in $(M_+,[\cdot,\cdot])$ and ${T_{X+Y,v}}^-$ be the synthesis operator for the frame $\{(X_i+Y_i,v_i):i\in{{I_-}}\}$ in $(M_-,-[\cdot,\cdot])$. Also let us assume that ${T^{\ast}_{X+Y,v}}^+$ and ${T^{\ast}_{X+Y,v}}^-$ be the adjoint operators of ${T_{X+Y,v}}^+$ and ${T_{X+Y,v}}^-$ respectively. Now let us define ${\mathfrak{T}_{X+Y,v}}^+:=~{T_{X,v}}^++{T_{Y,v}}^+$. Then we have ${\mathfrak{T}^{\ast}_{X+Y,v}}^+=~{T^{\ast}_{X,v}}^++{T^{\ast}_{Y,v}}^+$. A direct calculation shows that ${\mathfrak{T}_{X+Y,v}}^+={T_{X+Y,v}}^+$ is the synthesis operator for the frame $\{(X_i+Y_i,v_i):i\in{{I_+}}\}$ in $(M_+,[\cdot,\cdot])$.\\
Hence ${\mathfrak{T}^{\ast}_{X+Y,v}}^+{\mathfrak{T}_{X+Y,v}}^+=I$. Therefore, ${T^{\ast}_{X,v}}^+{T_{Y,v}}^++{T^{\ast}_{Y,v}}^+{T_{X,v}}^+=0$.\\
Similarly we can show that ${T^{\ast}_{X,v}}^-{T_{Y,v}}^-+{T^{\ast}_{Y,v}}^-{T_{X,v}}^-=0$.

Conversely we have to show that $\{(X_i,v_i):i\in{I}\}$ is a $\zeta-J$-Parseval fusion frame for $\mathbb{K}$. Since $\overline{\sum_{i\in{I_{+}}}(X_i+Y_i)}=M_+$. Hence at first we will show that $\{(X_i,v_i):i\in{I_+}\}$ is a Parseval fusion frame for $(M_+,[\cdot,\cdot])$. Let ${\mathfrak{T}_{X+Y,v}}^+$ be the synthesis operator for the Bessel family $\{(X_i,v_i):i\in{I_+}\}$. Then ${\mathfrak{T}_{X+Y,v}}^+(\{h_i\})=\sum_{i\in{I_+}}v_ih_i=\sum_{i\in{I_+}}v_i(f_i+g_i)={T_{X,v}}^++{T_{Y,v}}^+$. Now it is easy to show that the condition ${T^{\ast}_{X,v}}^+{T_{Y,v}}^++{T^{\ast}_{Y,v}}^+{T_{X,v}}^+=0$ implies that $\{(X_i,v_i):i\in{I_+}\}$ is a Parseval fusion frame for $M_+$. By similar arguments as above we can show that $\{(X_i,v_i):i\in{I_-}\}$ is a Parseval fusion frame for $M_-$. So we only need to find $\zeta\in{[\sqrt{2},2)}$ such that $\zeta=c_0(M_+,\verb"C")+c_0(M_-,\verb"C")$. Now since the quantities  $c_0(M_+,\verb"C")$ and $c_0(M_-,\verb"C")$ are fixed throughout and $\{(X_i,v_i):i\in{I}\}$ is a given $\zeta-J$-Parseval frame, so we already have $\zeta$ which satisfies the above equation.\\
Hence the proof.
\end{proof}
\begin{theorem}
Let $\{(W_i,v_i):i\in I\}$ be $J$-fusion frame for the Krein space $\mathbb{K}$ with cannonical $J$-dual fusion frame $\{(S_{W,v}^{-1}(W_i),v_i):i\in I\}$. Then $\forall~I_1\subset I$ and $\forall~f\in \mathbb{K}$ we have $\sum_{i\in{I_1}}\sigma_iv_i^2[\pi_{W_i}Jf,Jf]-\sum_{i\in{I}}\sigma_iv_i^2[\pi_{S_{W,v}^{-1}(W_i)}JS_{W,v}^{I_1},JS_{W,v}^{I_1}]\\
=\sum_{i\in{I_1^c}}\sigma_iv_i^2[\pi_{W_i}Jf,Jf]-\sum_{i\in{I}}\sigma_iv_i^2[\pi_{S_{W,v}^{-1}(W_i)}JS_{W,v}^{I_1^c},JS_{W,v}^{I_1^c}]$, where $\sigma_i=1$ if $i\in{I_+}$ and $\sigma_i=-1$ if $i\in{I_-}$.
\end{theorem}
\begin{proof}
Let $S_{W,v}$ denote the frame operator for $\{(W_i,v_i):i\in I\}$. Then we have $S_{W,v}(f)=\sum_{i\in I}\sigma_iv_i^2\pi_{J(W_i)}(f)$. Also $S_{W,v}=S_{W,v}^{I_1}+S_{W,v}^{I_1^c}$. Then $I=S_{W,v}^{-1}S_{W,v}^{I_1}+S_{W,v}^{-1}S_{W,v}^{I_1^c}$. From operator theory we have $S_{W,v}^{-1}S_{W,v}^{I_1}-S_{W,v}^{-1}S_{W,v}^{I_1^c}=S_{W,v}^{-1}S_{W,v}^{I_1}S_{W,v}^{-1}S_{W,v}^{I_1}-S_{W,v}^{-1}S_{W,v}^{I_1^c}S_{W,v}^{-1}S_{W,v}^{I_1^c}$. Then for every $f,g\in{\mathbb{K}}$ we have
\begin{equation*}
\begin{split}
[S_{W,v}^{-1}S_{W,v}^{I_1}(f),g]-& [S_{W,v}^{-1}S_{W,v}^{I_1}S_{W,v}^{-1}S_{W,v}^{I_1}(f),g] \\
&=[S_{W,v}^{I_1}(f),S_{W,v}^{-1}g]-[S_{W,v}^{-1}S_{W,v}^{I_1}(f),S_{W,v}^{I_1}S_{W,v}^{-1}g] 
\end{split}
\end{equation*}
Now if we choose $g=S_{W,v}(f)$, then the above equation reduces to 
\begin{equation*}
\begin{split}
=[S_{W,v}^{I_1}(f),f]-& [S_{W,v}^{-1}S_{W,v}^{I_1}(f),S_{W,v}^{I_1}(f)] \\
&=\sum_{i\in{I_1}}\sigma_iv_i^2[\pi_{W_i}Jf,Jf]-\sum_{i\in{I}}\sigma_iv_i^2[\pi_{S_{W,v}^{-1}(W_i)}JS_{W,v}^{I_1},JS_{W,v}^{I_1}] 
\end{split}
\end{equation*}
Now replacing $I_1$ by $I_1^c$ we can have the other part of the equality. Combining we finally get
\begin{equation*}
\begin{split}
\sum_{i\in{I_1}}\sigma_iv_i^2[\pi_{W_i}Jf,Jf]-& \sum_{i\in{I}}\sigma_iv_i^2[\pi_{S_{W,v}^{-1}(W_i)}JS_{W,v}^{I_1},JS_{W,v}^{I_1}] \\
& =\sum_{i\in{I_1^c}}\sigma_iv_i^2[\pi_{W_i}Jf,Jf]-\sum_{i\in{I}}\sigma_iv_i^2[\pi_{S_{W,v}^{-1}(W_i)}JS_{W,v}^{I_1^c},JS_{W,v}^{I_1^c}] 
\end{split}
\end{equation*}
\end{proof}
We need the following beautiful theorem for our further study. 
\begin{theorem}\cite{rgd}
Let $S$ and $T$ be bounded linear operators on a Hilbert space $\mathbb{H}$. Then the following conditions are equivalent\\
$(i)$ $R(S)\subset{R(T)}$\\
$(ii)$ there exists $\lambda\geq{0}$ such that $SS^{\ast}\leq\lambda^2{TT^{\ast}}$\\
$(iii)$ there exists a closed linear operator $S_1$ such that $SS_1=T$\\
\end{theorem}
Now we will characterize uniformly $J$-definite subspaces of a Krein space $\mathbb{K}$ in terms of an inequality regarding $J$-fusion frames.

Let $M$ be a non-negative subspace of a Krein space $\mathbb{K}$, then $M=M^0[\dot{+}]M^+$ where $M^+$ is the positive part of $M$ and $M^0$ is the isotropic part of $M$, precisely $M^+=M\ominus{M^0}$, where $M\ominus{M^0}=M\cap({M\cap{M^0}})^{\perp}$. Let $G_M$ be the Gramian operator of $M$ then $N(G_M)=M^0$ and $R(G_M)=M^+$. So we have $G_M=G_{M^+}$. The decomposition of $M$ is guaranteed by the spectral decomposition of the Gramian operator.
\begin{theorem}
	Let $\mathbb{F}=\{(W_i,v_i\}_{i\in I}$ be a Bessel family in a Krein space $\mathbb{K}$, also let ${M}=\overline{\sum{W_i} : i\in I}$ and $M^0=M{\cap}M^{[\perp]}$. If there exist constants $0<A\leq B$ such that
\begin{equation}\label{GJFFEQT}
	A [f,f] \leq \sum_{i\in I} v_i^2|[\pi_{JM}\pi_{W_i}J\pi_{M}(f),f]| \leq B [f,f] ~ \textmd{for every $f\in{M}$},
\end{equation}
	then the deficiency subspace $M\ominus M^0$ is a (closed) uniformly $J$-positive subspace of $M$. Moreover, if $\mathbb{F}$ is a frame for the Hilbert space $(M,[\cdot,\cdot]_J)$, the converse holds.
\end{theorem}

\begin{proof}
	First, let us assume that the equation (\ref{GJFFEQT}) holds. Then $[f,f]\geq{0}$ for all $f\in{M}$. So, $M$ is a $J$-nonnegative subspace of $\mathbb{K}$, or equivalently, $(M,[\cdot,\cdot])$ is a non-negative inner product space.
	
Now given that $\mathbb{F}=\{(W_i,v_i\}_{i\in I}$ is a Bessel family in $\mathbb{K}$. So it is a Bessel family in the associated Hilbert space $(\mathbb{K},[\cdot,\cdot]_J)$. Then $T_{W,v}^*(f)=\{v_i\pi_{W_i}(f)\}_{i\in{I}},~f\in{\mathbb{K}}$. Now for all $f\in{M}$,
\begin{equation*}
\sum_{i\in{I}}v_i^2[\pi_{W_i},f]_J\leq C[f,f]_J\\
\end{equation*}
where $C=\|T_{W,v}^*\|_J^2>0$, then $\sum_{i\in{I}}v_i^2\pi_{W_i}\leq C\pi_{M}$ since $\pi_{M}\pi_{W_i}=\pi_{W_i}\pi_{M}=\pi_{W_i}$. So, using equation (\ref{GJFFEQT}) it is easy to see that
\begin{equation*}
\begin{split}
A[{G_M f},{f}]_J &\leq \sum_{i\in I} v_i^2|[\pi_{JM}\pi_{W_i}J\pi_{M}(f),f]|=\sum_{i\in I} v_i^2|[\pi_{M}\pi_{JW_i}\pi_{M}(f),f]_J|\\
& =\sum_{i\in I} v_i^2|[\pi_{M}J\pi_{W_i}J\pi_{M}(f),f]_J|=|[\pi_{M}J\{\sum_{i\in I}v_i^2\pi_{W_i}\}J\pi_{M}(f),f]_J|\\
&\leq C [(G_M)^2 f,{f}]_J, ~f\in \mathbb{K}.
\end{split}	
\end{equation*}
Thus, $0\leq G_M\leq \frac{C}{A}\, (G_M)^2$. So applying Douglas theorem \cite{rgd} we have
\begin{equation*}
	R((G_M)^{1/2})\subseteq R(G_M)
\end{equation*}
Also we have $R(G_M)\subseteq R((G_M)^{1/2})$.
	Moreover, it follows that $R(G_M)$ is closed since $R(G_M)=R((G_M)^{1/2})$. 
	
	Let $M^{\prime}=M\ominus{M^0}$. Since $R(G_M)$ is closed, so by reduced minimum modulus theorem there exists $\gamma >0$ such that
\begin{equation*}
	[f,f]=[{G_M f},{f}]_J=\|(G_M)^{1/2} f\|_J^2 \geq \gamma \|f\|_J^2 ~\textmd{~for every $f\in N(G_M)^\perp=M\ominus M^0$}. 
\end{equation*}
Hence $M^{\prime}$ is a closed uniformly $J$-positive subspace of $\mathbb{K}$.
	
	Conversely, suppose that $W$ is a frame for $(M,[\cdot,\cdot]_J)$, i.e. there exist constants $B^{\prime}\geq A^{\prime}>0$ such that
\begin{equation*}
	A^{\prime} P_M \leq \sum_{i\in{I}}v_i^2\pi_{W_i} \leq B^{\prime} P_M,
\end{equation*}
\begin{equation*}
	A^{\prime} P_M \leq T_{W,v}T_{W,v}^\ast \leq B^{\prime} P_M,
\end{equation*}
	where $T_{W,v}\in L\big((\sum_{i\in{I}}\oplus{W_i})_{\ell_2},\mathbb{K}\big)$ is the synthesis operator of $W$. If $M^{\prime}=M\ominus{M^0}$ is a uniformly $J$-positive subspace of $\mathbb{K}$, then there exists some real number $\alpha>0$ such that $\alpha P_{M^{\prime}} \leq G_{M^{\prime}} \leq P_{M^{\prime}}$. Then by Douglas theorem, $R(P_{M'})\subseteq R((G_{M'})^{1/2})\subseteq R(P_{M^{\prime}})$. So we have $R({G_{M^{\prime}}}^{\frac{1}{2}})=M^{\prime}=R(G_{M^{\prime}})$. Since $G_{M}=G_{M^{\prime}}$ it is easy to see that
\begin{equation*}
A^{\prime} (G_{M})^2= A^{\prime}(G_{M^{\prime}})^2 \leq P_M JT_{W,v}T_{W,v}^\ast J P_M \leq B^{\prime}(G_{M^{\prime}})^2=B^{\prime}(G_{M})^2.
\end{equation*}
	Therefore again using Douglas theorem we have $R(P_M JT_{W,v})=R(G_{M^{\prime}})=R((G_{M^{\prime}})^{1/2})$ or equivalently, there exist $B\geq A>0$ such that
\begin{equation*}
A G_M= A G_{M^{\prime}} \leq P_M JT_{W,v}T_{W,v}^\ast J P_M \leq B G_{M^{\prime}}= B G_{M}
\end{equation*}
So from above we have
\begin{equation*}
 A[f,f] \leq [\pi_{M}J\{\sum_{i\in I}v_i^2\pi_{W_i}\}J\pi_{M}(f),f]_J \leq B[f,f] \textmd{~for every}~ f\in M
\end{equation*}
\begin{equation*}
 i.e.~~A[f,f] \leq \sum_{i\in I} v_i^2|[\pi_{M}\pi_{JW_i}\pi_{M}(f),f]_J| \leq B[f,f]
\end{equation*}
\begin{equation*}
 i.e.~~A[f,f] \leq \sum_{i\in I} v_i^2|[\pi_{JM}\pi_{W_i}J\pi_{M}(f),f]| \leq B[f,f] \textmd{~for every}~ f\in M
\end{equation*}
\end{proof}

{\bf Acknowledgements.} Shibashis Karmakar, Sk. Monowar Hossein and Kallol Paul gratefully acknowledge the support of Jadavpur University, Kolkata and Aliah University, Kolkata for providing all the facilities when the manuscript was prepared. Shibashis Karmakar also acknowledges the financial support of CSIR, Govt. of India.


\end{document}